\documentclass{amsart}

\usepackage{amsmath}
\usepackage{amsfonts}
\usepackage{amssymb}
\usepackage{amsthm}
\usepackage{amscd}

\usepackage{latexsym}
\usepackage{mathrsfs}
\usepackage{psfrag}
\usepackage[dvips]{epsfig}
\usepackage{epsfig}
\usepackage[all]{xy}        
\theoremstyle{plain}

\newtheorem{theorem}{Theorem}[section]
\newtheorem{corollary}[theorem]{Corollary}
\newtheorem{proposition}[theorem]{Proposition}
\newtheorem{lemma}[theorem]{Lemma}

\theoremstyle{definition}
\newtheorem{definition}[theorem]{Definition}

\newtheorem{example}[theorem]{Example}
\newtheorem{examples}[theorem]{Examples}

\newtheorem{remark}[theorem]{Remark}

\numberwithin{equation}{section}

\usepackage{orcidlink}

\DeclareMathOperator{\Id}{Id}
\DeclareMathOperator{\Seq}{Seq}
\DeclareMathOperator{\Cauchy}{Cauchy}
\DeclareMathOperator{\Conv}{Conv}
\DeclareMathOperator{\Zero}{Zero}

\begin{document}
	
	\title{Magma-valued metric spaces}
	
	
	\author{Peyman Nasehpour\orcidlink{0000-0001-6625-364X}}
	
	\curraddr{Academic Advisor and Education Mentor\\ Education Department\\ The New York Academy of Sciences\\ New York, NY, USA}
	
	
	\email{nasehpour@gmail.com}
	
	
	\subjclass{06F05, 16Y30, 40A05.}
	\keywords{dense monoid, ordered near-rings, convergent sequences, metric spaces}
	
	\begin{abstract}
		In the second section, we introduce dense unital magmas and show that a near-ring is dense if and only if it has a positive element smaller that unity. In the third section, we discuss magma-valued metric spaces. The density property of the ordered unital magmas and monoids helps us to generalize a couple of classical results related to the convergence and Cauchy property of sequences. The last section is devoted to magma-valued normed groups and some subgroups of the additive group of Cauchy sequences.  
	\end{abstract}

	\maketitle

	\section{Introduction}
	
	In 1906, Fr\'{e}chet \cite{Frechet1906} formulated the axioms of metric spaces for the first time in his dissertation. Following Fr\'{e}chet's pioneering work, the concept of distance function has undergone numerous generalizations leading to the development of various mathematical fields. For example, in his 1934 work \cite{Kurepa1934}, Kurepa modified the axioms of metric spaces by removing the positivity requirement allowing for the possibility that a pseudometric space could satisfy the condition $d(x, y) = 0$ for $x \neq y$. Later in 1942, Menger introduced statistical metric spaces \cite{Menger1942} (check also \cite{SchweizerSklar1960}). The main purpose of this paper is to introduce the concept of magma-valued metric spaces and generalize some classical results of abstract mathematical analysis in this context.
	
	First we introduce some terminologies and notations, and then briefly report what we do in the rest of the paper. It is said that $(M,*)$ is a magma if $M$ is a nonempty set and $*: M \times M \rightarrow M$ is a function \cite[Definition 1.1]{Serre1992}. Let $(M,*)$ be a magma. It is said that a (binary) relation $R$ on $M$ is compatible with the magma $(M,*)$ if $a R b$ implies \[(a*c) R (b*c) \text{~and~} (c*a) R (c*b), \qquad \forall~a,b,c \in M.\] 
	
	Let us recall that a relation on a set is partial if it is reflexive, antisymmetric, and transitive. A relation $R$ on a set $M$ is linear if either $a R b$ or $b R a$ for all $a,b \in M$. In this paper, partial relations, also known as partial orderings, are denoted by $\leq$. If a partial ordering $\leq$ on a set $M$ is linear, it is said that $\leq$ is a total ordering on $M$ or $(M,\leq)$ is totally ordered.
	
	If $\leq$ is a partial ordering on $M$, then $(M,*,\leq)$ is said to be an ordered magma if $\leq$ is compatible with the magma $(M,*)$. As usual, if $\leq$ is a relation on $M$, then by $a < b$ it is meant that $a \leq b$ and $a \neq b$. Note that by the sentence ``$(M,*,<)$ is an ordered magma'' ($M$ is a magma or any other group-like algebraic structure), we mean that $<$ is compatible with $*$. This already means that if $a < b$, then $a*c < b*c$ and $c*a < b*a$, for all $a,b,c \in M$. 
	
	A magma $(M,*)$ is a unital magma if there is an element $0 \in M$ such that \[m*0 = 0*m = m, \qquad \forall~m \in M.\] In any ordered unital magma $(M,*,0,\leq)$, an element $m \in M$ is non-negative, if $0 \leq m$. Also, $m \in M$ is positive, if $0 < m$, i.e. $0\leq m$ and $m \neq 0$.
	
	In \S\ref{sec:densemagmas}, we introduce a property called density and find some examples for that. We define an ordered unital magma $(M,*,0,\leq)$ to be {\em dense} (see Definition \ref{densemagmadef}) if 
	
	\begin{itemize}
		\item for any positive element $\epsilon$ in $M$, there are two positive elements $\beta$ and $\gamma$ with $\beta*\gamma < \epsilon$.
	\end{itemize}
	
	Now, we proceed to report some results of the paper that we present in \S\ref{sec:magmavaluedmetricspaces} in order to justify our definition for density. 
	
	Let $(M,*,0,\leq)$ be an ordered unital magma. As a generalization to monoid-valued metric spaces (see \cite{Braunfeld2017} and \cite{Conant2019}), we say $(X,d)$ is an $M$-metric space (see Definition \ref{magmavaluedmetricspacesdef}) if $d: X \times X \rightarrow M$ is a function such that for all $x,y,z \in X$, the following properties hold:
	
	\begin{enumerate}
		\item $d(x,y) \geq 0$, and $d(x,y) = 0$ if and only if $x = y$.
		\item $d(x,y) = d(y,x)$.
		\item $d(x,z) \leq d(x,y) * d(y,z)$.
	\end{enumerate}
	
	By definition, a sequence $(x_n)$ in an $M$-metric space $X$ is convergent to $x \in X$, denoted by \[\lim_{n \to +\infty} x_n = x,\] if for any positive element $\epsilon$ in $M$, there is a natural number $N$ such that $n \geq N$ implies $d(x_n,x) < \epsilon$.
	
	In Proposition \ref{limitsumofsequences}, we prove that if $(M,*,0,\leq)$ is a dense unital magma and $X$ is an $M$-metric space, then a convergent sequence has a unique limit. Not only that but also density property is useful for generalizing a couple of other classical results in abstract analysis. For example, with the help of density, we prove that Cauchyness is implied by convergence as we explain in the following:
	
	Let $M$ be an ordered unital magma and $X$ an $M$-metric space. It is natural to define that a sequence $(x_n)$ in $X$ is a Cauchy sequence if for any positive element $\epsilon$ in $M$ there is a natural number $N$ such that $m,n \geq N$ implies $d(x_m, x_n) < \epsilon$ (see Definition \ref{Cauchysequencedef}). Next, in Theorem \ref{convergentisCauchy}, we prove that if $M$ is a dense unital magma and $X$ an $M$-metric space, then any convergent sequence in $X$ is a Cauchy sequence. Also in Theorem \ref{CauchyconvergentsubsequenceOF}, we show that if $M$ is a dense unital magma and $X$ is an $M$-metric space, then any Cauchy sequence having a convergent subsequence is convergent.
	
	Surprisingly, the definition of density property given above is equivalent to the definition of density property for ordered rings given in \cite{Heuer1974}, and also the definition of density in set theory given in \cite{Jech2003}. As a matter of fact, we prove that a totally ordered ring with 1 is dense if and only if its additive monoid is dense. Even more, we generalize this for near-rings. Recall that an algebraic structure $(N,+,\cdot,0)$ is a near-ring (see Definition 1.1 in \cite{Pilz1983}) if the following conditions are satisfied:
	
	\begin{enumerate}
		\item $(N,+)$ is a (not necessarily abelian) group and 0 is the identity element of the group $N$.
		\item $(N,\cdot)$ is a semigroup.
		\item The right-distributive law holds, i.e. \[(x+y)z = xz + yz, \qquad \forall~x,y,z \in N.\] 
	\end{enumerate}
	
	Note that, by definition, a near-ring $N$ is with 1 if $(N,\cdot,1)$ is a monoid.
	
	It is said that $(N,+,\cdot,0,\leq)$ is an ordered near-ring (see Definition 9.122 in \cite{Pilz1983}) if $\leq$ is a partial ordering on $N$ and $(N,+,\cdot,0)$ is a near-ring such that the following properties hold:
	
	\begin{enumerate}
		\item $(N,+,0,\leq)$ is an ordered group.
		\item If $0 \leq x$ and $0 \leq y$, then $0 \leq xy$, for all $x,y \in N$. 
	\end{enumerate}
	
	We say $(N,+,\cdot,0,<)$ - for short, $(N,<)$ - is an ordered near-ring if $(N,\leq)$ is an ordered near-ring and $0 < x$ and $0<y$ imply $0 < xy$, for all $x,y \in N$.
	
	In Theorem \ref{denserings}, we prove that if $\leq$ is a total ordering on $N$ and $(N,<)$ is an ordered near-ring with $1\neq 0$, then the following statements are equivalent:
	
	\begin{enumerate}
		\item The near-ring $N$ is dense, i.e. the positive cone $P$ of $N$ has no least element \cite{Heuer1974}.
		\item There is a positive element $\alpha$ in $N$ smaller than 1.
		\item The ordered monoid $(N,+,0,\leq)$ is a dense monoid.
		\item For any positive element $\epsilon$ in $N$ and a positive integer number $n$, we can find $n$ positive elements $\{\epsilon_i\}^n_{i=1}$ in $N$ satisfying the following inequality: \[\sum_{i=1}^{n} \epsilon_i < \epsilon.\]
		
		\item The ordered set $N$ is dense, i.e. for $r < t$ in $N$ there is an $s$ in $N$ with $r < s < t$ \cite[Definition 4.2]{Jech2003}.
	\end{enumerate}
	
	In \S\ref{sec:normedgroups}, we introduce magma-valued norms in the following way (see Definition \ref{Mnormedgroupdef}):
	
	Let $(M,+,0,\leq)$ be an ordered unital magma. A group $(G,+)$ is, by definition, an $M$-normed group if there is a function $\Vert \cdot \Vert: G \rightarrow M$ with the following properties:
	\begin{enumerate}
		\item $\Vert g \Vert \geq 0$, and $\Vert g \Vert = 0$ if and only if $g = 0$, for all $g\in G$.
		\item $\Vert g-h \Vert \leq \Vert g \Vert + \Vert h \Vert$, for all $g,h \in G$.
	\end{enumerate}
	
	These norms evidently induce magma-valued metric spaces (check Proposition \ref{inducedmagmavaluedmetricspace}). By considering this, in Theorem \ref{Convergentsequencesabeliangroup}, we prove that if $M$ is a dense unital magma and $G$ an $M$-normed group, then $\Zero(G)$ is a subgroup of $\Conv(G)$ and \[\Conv(G)/\Zero(G) \cong G,\] where $\Conv(G)$ is the group of all convergent sequences in $G$ and $\Zero(G)$ is the set of all sequences in $\Conv(G)$ convergent to $0 \in G$.
	
	In this paper, a binary operation of a magma $M$ is sometimes denoted by ``+'', although ``+'' is not necessarily associative or commutative unless explicitly stated. Also, if the binary operation of a monoid or a group is denoted additively, similar to near-ring theory \cite{Pilz1983}, it does not mean that the addition is necessarily commutative unless explicitly stated. Whenever we say an algebraic structure is ordered, we mean that its ordering is partial, otherwise we explicitly assert that the ordering is total (linear). For the general theory of ordered algebraic structures see \cite{Fuchs1963}. For ordered groups consult with \cite{Glass1999}. Some results in this paper are generalizations of their counterparts on pages 18--24 in Ovchinnikov's 2021 book \cite{Ovchinnikov2021}.
	
	\section{Dense magmas}\label{sec:densemagmas}
	
	\begin{definition}\label{densemagmadef}
		Let $(M,*,0,\leq)$ be a unital magma. We say $M$ is a dense unital magma if for any positive element $\epsilon$ in $M$, there are two positive elements $\beta$ and $\gamma$ with $\beta*\gamma < \epsilon$.
	\end{definition}
	
	\begin{remark}
		Let $(M,*,0,\leq)$ be an ordered group-like algebraic structure such that 0 is an identity element of $M$. We say $M$ is dense if $(M,*,0,\leq)$ as an ordered unital magma is dense.
	\end{remark}
	
	\begin{examples} In the following, we give a couple of examples:
		
		\begin{enumerate}
			
			\item Let $K$ be any subfield of the field of real numbers. Evidently, $(K,+,0,\leq)$ is a dense unital magma because for any positive number $\epsilon \in K$, we have \[2\epsilon/5 + 2\epsilon/5 < \epsilon.\]
			
			\item The additive group of integer numbers $(\mathbb Z,+,0,\leq)$ is not dense.
			
			\item Let $M$ be the set of all functions of the form $f: \mathbb R \rightarrow \mathbb R$. Let addition be component-wise on $M$ and define $f \leq g$ in $M$, if $f(x) \leq g(x)$, for all $x\in \mathbb R$. Then, $M$ is a dense unital magma because for the given $0 < f$, the function $g = 2f/5$ satisfies $0 < g$ and \[g + g < f.\]
		\end{enumerate}
	\end{examples}
	
	Let us recall that if $(P,\leq)$ is a poset with no least element, one may annex an element $-\infty$ to $P$ and extend $\leq$ as follows: \[-\infty < x \qquad \forall~x \in P.\] Then, $-\infty$ is the least element of the new poset $(P \cup \{-\infty\}, \leq)$.
	
	\begin{proposition}\label{maxmonoids}
		Let $(T,\leq)$ be a totally ordered set with no least element. Annex the least element $-\infty$ to $T$. Then, $(T\cup \{-\infty\}, \max)$ is a dense monoid.	
	\end{proposition}
	
	\begin{proof}
		It is evident that $(T\cup \{-\infty\}, \max)$ is a commutative monoid and its neutral element is $-\infty$. Note that by definition, for any $a\in T$, we have $-\infty < a$. It is also easy to see that $a \leq b$ implies $\max\{a,c\} \leq \max\{b,c\}$ for all $a,b,c \in T\cup \{-\infty\}$. Therefore, $T\cup \{-\infty\}$ is an ordered monoid. 
		
		Since $T$ has no least element, for the given element $\epsilon \in T$, there is an element $\beta \in T$ such that $\beta < \epsilon$. Now, if we set $\gamma = \beta$, we see that $\max\{\beta,\gamma\} = \beta < \epsilon.$ Hence, $T$ is a dense monoid, as required.
	\end{proof}
	
	\begin{remark}
		The monoid $(\mathbb R \cup \{-\infty\}, \max)$ which is an example of a dense monoid (Proposition \ref{maxmonoids}) is extensively used in idempotent analysis. For more, refer to \cite{KolokoltsovMaslov1997} and \cite{LMS2002}.
	\end{remark}
	
	Let $G$ be an ordered group. We say $G$ is a dense group if the monoid $(G,+,0,\leq)$ is dense. In the following, we give an example of a dense but non-commutative monoid:
	
	\begin{theorem}\label{totallyorderednonabeliangroup}
		The set $\mathbb R \times \mathbb R$ equipped with the following binary operation and the lexicographical ordering is a non-abelian totally ordered dense group: \[(r_1,r_2)*(s_1,s_2) = (r_1 + s_1, r_2 e^{s_1} + s_2).\]
	\end{theorem}
	
	\begin{proof}
		It is routine to see that $(\mathbb R \times \mathbb R, *,(0,0), <)$ is a non-abelian totally ordered group, where $<$ is the lexicographical ordering on $\mathbb R \times \mathbb R$ (see p. 140 in \cite{vanOystaeyen2000}). Let $(r,s)$ be a positive element of $\mathbb R \times \mathbb R$, i.e. $(0,0) < (r,s)$. Then, we have two cases: Either $0 < r$, or $0 = r$ and $0 < s$.
		
		Assume that $0 < r$. So, there are two positive real numbers $r_1$ and $r_2$ with $r_1 + r_2 < r$. Now, it is obvious that \[(r_1,s) * (r_2,s) < (r,s).\]
		
		Assume that $0 = r$ and $0 < s$. Therefore, there are two positive real numbers $s_1$ and $s_2$ such that $s_1 + s_2 < s$. Observe that \[(0,s_1) < (0,s),~(0,s_2) < (0,s) \] and \[(0,s_1) * (0,s_2) = (0,s_1 + s_2) < (0,s).\] Hence, $\mathbb R \times \mathbb R$ is dense and the proof is complete.
	\end{proof}
	
	\begin{lemma}\label{densemonoidlem}
		Let $(M,+,0,\leq)$ be a dense (not necessarily commutative) monoid. Then, for each $\epsilon > 0$ and $n \in \mathbb N$, there are $n$ positive elements $\epsilon_i$ in $M$ with \[\sum_{i=1}^{n} \epsilon_i < \epsilon.\]
	\end{lemma}
	
	\begin{proof}
		The proof is by induction on $n$. By Definition \ref{densemagmadef}, the cases $n=1$ and $n=2$ hold evidently. Now, assume that $n = k$ holds. Therefore, for the given $\epsilon > 0$, we have \[\sum_{i=1}^{k} \epsilon_i < \epsilon.\] Since $\epsilon_k$ is positive, we can find two positive elements $\beta_1$ and $\epsilon_{k+1}$ such that \[\beta_1 + \epsilon_{k+1} < \epsilon_k.\] Now, observe that \[\epsilon_1 + \dots + \epsilon_{k-1} + \beta_1 + \epsilon_{k+1} \leq \sum_{i=1}^{k}\epsilon_i < \epsilon.\] This completes the proof.
	\end{proof}
	
	Let $\leq$ be a total ordering on $R$ and $(R,<)$ an ordered ring with 1. It is easy to see that the positive cone $P$ of $R$ has the least element if and only if 1 is the least element of $P$ \cite{Heuer1974}. Heuer \cite{Heuer1974} calls an ordered ring discrete if its positive cone has the least element; otherwise dense. Note that in order theory, a totally ordered set $(S,<)$ is dense if for each $a,b \in S$ with $a < b$ there is an element $c\in S$ such that $a < c < b$ \cite[Definition 4.2]{Jech2003}. The following result supports the use of the term ``dense unital magma'' in Definition \ref{densemagmadef}:
	
	\begin{theorem}\label{denserings}
		Let $\leq$ be a total ordering on $N$. Also, let $(N,<)$ be an ordered near-ring with $1\neq 0$. Then, the following statements are equivalent:
		
		\begin{enumerate}
			\item The near-ring $N$ is dense, i.e. the positive cone $P$ of $N$ has no least element.
			\item There is a positive element $\alpha$ in $N$ smaller than 1.
			\item The ordered monoid $(N,+,0,\leq)$ is a dense monoid.
			\item For any positive element $\epsilon$ in $N$ and a positive integer number $n$, we can find $n$ positive elements $\{\epsilon_i\}^n_{i=1}$ in $N$ satisfying the following inequality: \[\sum_{i=1}^{n} \epsilon_i < \epsilon.\]
			
			\item The ordered set $N$ is dense, i.e. for $r < t$ in $N$ there is an $s$ in $N$ with $r < s < t$.
		\end{enumerate}
	\end{theorem}
	
	\begin{proof}
		$(1) \implies (2)$: Since $\leq$ is a total ordering on the near-ring $N$, by Remark 9.134 in \cite{Pilz1983}, 1 is a positive element of $N$. The set of positive elements of $N$ has no least element. Therefore, there must be an element between 0 and 1.
		
		$(2) \implies (3)$: Let there be an $\alpha \in N$ with $0 < \alpha < 1$. Since $(N,<)$ is an ordered near-ring, we have \[0 < -\alpha + 1.\] This implies that $0 < (-\alpha + 1) \alpha$. By the right-distributive law, we have $0 < -\alpha^2 + \alpha$. On the other hand, since $0 < \alpha$, we have $0 < \alpha^2$. Now, let the positive element $\epsilon$ be given and set \[\epsilon_1 = \alpha^2 \epsilon \text{~and~} \epsilon_2 = (-\alpha^2+\alpha) \epsilon.\] Observe that $\epsilon_1$ and $\epsilon_2$ are positive elements and in view of Proposition 1.5 in \cite{Pilz1983}, we have \[\epsilon_1 + \epsilon_2 = \alpha \epsilon < \epsilon.\]
		
		$(3) \implies (4)$: Lemma \ref{densemonoidlem}.
		
		$(4) \implies (1)$: For any positive element of $N$, we can find a smaller positive element in $N$. This means that the positive cone of $N$ has no least element.
		
		Up to now, we have proved that the statements (1), (2), (3), and (4) are equivalent. Now, we show that the statement (5) is also equivalent to any of them:
		
		$(5) \implies (2)$: Since $0 < 1$, we can find an $\alpha \in N$ with $0 < \alpha < 1$.
		
		$(4) \implies (5)$: Let $r < t$. This implies that $0 < -r+t$. So, by assumption, there is a positive element $\epsilon$ in $N$ with \[0 < \epsilon < -r+t.\] Evidently this implies that $r < r + \epsilon < t$ showing that $(N,<)$ is a dense ordered set and the proof is complete.
	\end{proof}
	
	\begin{corollary}
		Let $\leq$ be a total ordering on $N$. Also, let $(N,<)$ be an ordered near-ring with $1 \neq 0$. Assume that $S$ is a sub-near-ring of $N$ such that $1\in S$. If $S$ is a dense near-ring, then so is the near-ring $N$. 
	\end{corollary}
	
	\begin{proof}
		Since $S$ is a dense near-ring, by Theorem \ref{denserings}, there is an element $\alpha$ in $S$ with $0 < \alpha < 1$. Since $S \subseteq N$, we have $\alpha \in N$. Thus by Theorem \ref{denserings}, $N$ is also a dense near-ring and the proof is complete.
	\end{proof}
	
	\begin{corollary}\label{tofielddense}
		Let $(F,<)$ be a totally ordered field. Then, $(F,+,0,\leq)$ is a dense monoid.
	\end{corollary}
	
	Let $(G,+)$ be a group and 0 be its neutral element. Assume that $M_0(G)$ is the set of all functions $f$ from $G$ into $G$ with $f(0) = 0$. Then, $M_0(G)$ equipped with component-wise addition and composition of functions is a near-ring (see Example 1.4 in \cite{Pilz1983}). It is clear that the set of all increasing functions $\mathcal{I}_0(\mathbb R)$ in $M_0(\mathbb R)$ is a sub-near-ring of $M_0(\mathbb R)$. 
	
	\begin{proposition}
		Define $\leq$ on $M_0(\mathbb R)$ as follows: \[f \leq g \text{~~if~~} f(x) \leq g(x),~\forall~x \in \mathbb R.\] Then, $(\mathcal{I}_0(\mathbb R),\leq)$ is a dense ordered near-ring with $0 < I$, where by $I$, we mean the identity function on $\mathbb R$.
	\end{proposition}
	
	\begin{proof}
		It is easy to show that $\leq$ is a partial ordering on $\mathcal{I}_0(\mathbb R)$ and $f \leq g$ implies that $f+h \leq g+h$, for all $f$, $g$, and $h$ in $\mathcal{I}_0(\mathbb R)$. In order to prove that $(\mathcal{I}_0(\mathbb R),\leq)$ is an ordered near-ring we need to prove that if $f\geq 0$ and $g \geq 0$, then $f \circ g \geq 0$. By assumption, we have $g(x) \geq 0$ for all $x\in \mathbb R$. Since $f$ is an increasing function and passes through origin, i.e. $f(0) = 0$, we have \[(f \circ g)(x) = f(g(x)) \geq f(0) = 0.\] Now, we proceed to prove that $\mathcal{I}_0(\mathbb R)$ is dense. Observe that if $0 < f$, then there is at least one point $x$ in $\mathbb R$ such that $0 < f(x)$. Set $g = f/3$. It is evident that $0 < g$ and \[g + g  = 2f/3 < f.\] This completes the proof.
	\end{proof}
	
	\begin{definition}\label{DeMarrdivisionring}
		Let $(D,+,\cdot,0,1)$ be a division ring and $\leq$ a partial ordering on $D$. We say $(D,\leq)$ is a DeMarr division ring if the following conditions are satisfied:
		\begin{enumerate}
			\item $(D,+,0,\leq)$ is an ordered monoid.
			\item If $0 \leq x$ and $0 \leq y$, then $0 \leq xy$, for all $x,y \in D$.
			\item $0 < 1$.
			\item If $0 < x$, then $0 < x^{-1}$, for all $x\in D$. 
		\end{enumerate}  
	\end{definition}
	
	\begin{theorem}\label{DeMarrdivisionringdense}
		Let $(D,\leq)$ be a DeMarr division ring. Then, $(D,+,0,\leq)$ is a dense monoid.
	\end{theorem}
	
	\begin{proof}
		By definition, $0 < 1$. Set $n = n \cdot 1 = \sum_{i=1}^{n} 1$. This implies that $0 < n$ and also, $0 < n^{-1}$, for each positive integer $n$. Evidently, since $D$ is a division ring, we obtain that \[0 < m\cdot n^{-1} \text{~and~} 0 < n^{-1} \cdot m \qquad( m,n \in \mathbb N).\] In particular, $2\cdot 5^{-1}$ is positive. On the other hand, since $4\cdot 1 < 5 \cdot 1$, by multiplying both sides of the inequality by $5^{-1}$, we obtain that \[2\cdot 5^{-1} + 2\cdot 5^{-1} = (2+2)\cdot 5^{-1} = 4\cdot 5^{-1} < 1.\]  
		
		Now, assume that $0 < x$ in $D$ is given. Set $y = z = (2 \cdot 5^{-1}) x$ and observe that $y$ and $z$ are positive and \[0 < y + z = (4\cdot 5^{-1}) x < x.\] Hence, $(D,+,0,\leq)$ is a dense monoid, as required.
	\end{proof}
	
	In the following, we give an example of a DeMarr field which is not a totally ordered field. 
	
	\begin{example}
		Define $\leq$ on the field of complex numbers $\mathbb C$ as $z_1 \leq z_2$ if $z_2 - z_1$ is a non-negative real number. Then, $(\mathbb C, \leq)$ is a DeMarr field which is not a totally ordered field \cite[Example III]{DeMarr1967}.
	\end{example}
	
	Let $(R,<)$ be an ordered ring with 1, $P$ its positive cone, and $M$ an $R$-module. It is said that $M$ is an ordered $R$-module ordered by $N$ if $N$ is a nonempty subset of $M$ satisfying the following properties:
	
	\begin{itemize}
		\item If $m_1, m_2 \in N$ then $m_1 + m_2 \in N$, for all $m_1,m_2 \in M$.
		\item $N \cap -N = \{0\}$, where $-N = \{-x: x\in N\}$. 
		\item If $r \in R \setminus -P$ and $m \in N$, then $rm \in N$.
		
	\end{itemize}
	
	\begin{proposition}
		Let $(R,\leq)$ be an ordered ring, and $M$ an ordered $R$-module ordered by $N$. Define $\leq$ on $M$ by $m_2 \leq m_1$ if $m_1 - m_2 \in N$. Then, the following statements hold:
		\begin{enumerate}
			\item $(M,\leq)$ is a partially ordered set.
			\item If $x \leq y$ then, $x+z \leq y+z$ for all $x,y,z \in M$.
			\item If $x \leq y$ and $0 \leq r$, then $rx \leq ry$, for all $r\in R$ and $x,y \in M$.
		\end{enumerate}
	\end{proposition}
	
	\begin{proof}
		Straightforward.
	\end{proof}
	
	Let us recall that if $R$ is an integral (commutative) domain and $M$ a unital $R$-module, then by definition, $M$ is said to be a torsion-free $R$-module if $rm = 0$ implies either $r = 0$ or $m = 0$ for all $r\in R$ and $m\in M$ \cite[p. 134]{Rotman2009}.
	
	\begin{theorem}\label{densemodules}
		Let $\leq$ be a total ordering on $R$. Also, let $(R,<)$ be an ordered dense ring and $M$ an ordered torsion-free $R$-module. Then, $(M^{\geq 0},+,0,\leq)$ is a dense monoid, where by $M^{\geq 0}$, we mean the set of non-negative elements of $M$.
	\end{theorem}
	
	\begin{proof}
		Since $R$ is a dense ring, by Theorem \ref{denserings}, there is an element $\alpha \in R$ with $0 < \alpha < 1$. It is, then, clear that $0 < \alpha^2 < \alpha < 1$. Take $m \in M$ with $0 < m$. Since $M$ is ordered and torsion-free, we obtain that \[0 < \alpha^2 m < \alpha m < m.\] Similarly, since $0 < \alpha - \alpha^2 < 1$, we have $0 < (\alpha-\alpha^2)m < m$. Set \[m_1 = \alpha^2 m \text{~and~} m_2 = (\alpha-\alpha^2)m.\] Observe that $0 < m_1 < m$ and $0 < m_2 < m$ and we have \[m_1 + m_2 = \alpha^2 m + (\alpha-\alpha^2)m = \alpha m < m.\] Hence, $(M^{\geq 0},+,0,\leq)$ is a dense monoid, as required.
	\end{proof}
	
	\begin{example}
		Let $C[0,1]$ be the set of all continuous functions from $[0,1]$ into $\mathbb R$. Set \[N = \{f\in C[0,1]: f(x)\geq 0, \forall~x\in [0,1]\}.\] It is easy to see that $C[0,1]$ is an ordered $\mathbb R$-vector space ordered by $N$. By Theorem \ref{densemodules}, $M^{\geq 0}$ is dense.
	\end{example}
	
	\begin{example}
		Let $R$ be a commutative ring with 1. Then, $(\Id(R),+,\cdot, \subseteq)$ is an ordered semiring \cite[Example 1.4]{Golan1999(b)}, where the addition and multiplication of ideals are defined as follows: \[I+J = \{a+b: a\in I,b\in J\}, \text{~and~}\] \[IJ = \left\{\sum_{i=1}^{n} a_ib_i: a_i \in I, b_i\in J, i \in \mathbb N\right\}.\] Note that $(\Id(R),+,(0),\subseteq)$ is never a dense monoid because the set of maximal ideals of $R$ is always nonempty and there is no ideal properly between a maximal ideal $\mathfrak{m}$ of $R$ and the ring $R$.
	\end{example}
	
	\section{Magma-valued metric spaces}\label{sec:magmavaluedmetricspaces}
	
	As a generalization of monoid-valued metric spaces \cite{Conant2019} (see also Definition 2.10 in \cite{NasehpourParvardi2018}), we introduce magma-valued metric spaces as follows:
	
	\begin{definition}\label{magmavaluedmetricspacesdef}
		Let $(M,*,0,\leq)$ be an ordered unital magma. We say $(X,d)$ is an $M$-metric space if $d: X \times X \rightarrow M$ is a function such that for all $x,y,z \in X$, the following properties hold:
		
		\begin{enumerate}
			\item $d(x,y) \geq 0$, and $d(x,y) = 0$ if and only if $x = y$.
			\item $d(x,y) = d(y,x)$.
			\item $d(x,z) \leq d(x,y) * d(y,z)$.
		\end{enumerate}
	\end{definition}
	
	\begin{remark}
		In \cite{Braunfeld2017}, $L$-metric spaces have been discussed, where $L$ is a complete lattice. Note that if $X$ is an $L$-metric space, the triangle inequality is as follows: \[d(x,y) \leq d(x,z) \vee d(z,y), \qquad \forall~x,y,z \in X.\] We recall that a lattice $L$ is complete if, by definition, any subset $P$ of $L$ has infimum and supremum (see Definition 4.1 in \cite{BurrisSankappanavar1981}). It is obvious that if $L$ is complete and $0 = \inf(L)$, then 0 is the least element of $L$ and identity element of $\vee$, so that $(L,\vee,0,\leq)$ is an ordered monoid.
	\end{remark}
	
	\begin{proposition}
		Let $(A,\oplus,\neg)$ be an MV-algebra. Then, $(A,d)$ is an $A$-metric space, where \[d(x,y) = (x \ominus y) \oplus (y \ominus x), \qquad \forall~x,y \in A.\]
	\end{proposition}
	
	\begin{proof}
		By Definition 1.1.1, Lemma 1.1.2, and Lemma 1.1.4 in \cite{CDM2000}, $(A,\leq)$ is an ordered monoid such that $0$ is its smallest element. By Proposition 1.2.5 in \cite{CDM2000}, $(A,d)$ is an $A$-metric space and the proof is complete.
	\end{proof}
	
	\begin{proposition}\label{Mmetricspaces}
		Let $(G,+,\leq)$ be a totally ordered (but not necessarily abelian) group. Define $|\cdot|: G \rightarrow G$ by $|x| = \max\{x,-x\}$ and set $d(x,y)= |x-y|$. Then, $(G,d)$ is a $G$-metric space.
	\end{proposition}
	\begin{proof}
		The proof is similar to the case of the absolute value function over real numbers (see \cite{Swanson2021}) and so, omitted.
	\end{proof}
	
	\begin{example}
		An example of a non-abelian group satisfying the condition of Proposition \ref{Mmetricspaces} is the group explained in Theorem \ref{totallyorderednonabeliangroup}.
	\end{example}
	
	\begin{proposition}
		Let $(M,+,0,\leq)$ be an ordered commutative monoid and $(X_i,d_i)$ be an $M$-metric space for each $1 \leq i \leq n$. Define a function \[d: \prod_{i=1}^{n} X_i \times \prod_{i=1}^{n} X_i \longrightarrow M\] by \[d\big((x_i)_{i=1}^{n},(y_i)_{i=1}^n\big) = \sum_{i=1}^n d_i(x_i,y_i).\] Then, $(\prod_{i=1}^{n} X_i,d)$ is an $M$-metric space.
	\end{proposition}
	
	\begin{proof}
		Straightforward.
	\end{proof}
	
	\begin{definition}
		A sequence $(x_n)$ in an $M$-metric space $X$ is convergent to $x \in X$, denoted by \[\lim_{n \to +\infty} x_n = x,\] if for any positive element $\epsilon$ in $M$, there is a natural number $N$ such that $n \geq N$ implies $d(x_n,x) < \epsilon$.
	\end{definition}
	
	\begin{proposition}\label{limitsumofsequences}
		Let $(M,*,0,\leq)$ be a dense unital magma. In an $M$-metric space $X$, if a sequence is convergent in $X$, then its limit is unique.
	\end{proposition}
	
	\begin{proof}
		Let $(x_n)$ be convergent to $a$ and $b$. If $a\neq b$, then $d(a,b)> 0$. Set $\epsilon = d(a,b)$. Since $M$ is dense, we can find two positive elements $\beta$ and $\gamma$ such that $\beta * \gamma < \epsilon$. For $\beta$, we find a natural number $N_1$ such that if $n \geq N_1$, then $d(x_n,a) < \beta$. For $\gamma$, we find a natural number $N_2$ such that if $n \geq N_2$, then $d(x_n,b) < \gamma$. Now, set $N = \max\{N_1,N_2\}$ and observe that \[\epsilon = d(a,b) \leq d(x_n,a) * d(x_n,b) \leq \beta * \gamma < \epsilon,\] a contradiction. Thus the limit of any sequence is unique if it exists and the proof is complete.
	\end{proof}
	
	\begin{example}\label{limitofconstantsequences}
		Let $(M,*,0,\leq)$ be an ordered unital magma and $X$ an $M$-metric space and assume that there is a natural number $N$ such that $x_n = c \in X$ for all $n \geq N$. Then, $(x_n)$ is convergent to $c$. In particular, a constant sequence converges to its constant value.
	\end{example}
	
	\begin{proposition}\label{limitshiftrule} 
		Let $k$ be a natural number and $(x_n)$ a sequence in an $M$-metric space $X$. Then, $(x_n)$ is convergent if and only if $(x_{n+k})$ is convergent.
	\end{proposition}
	
	\begin{proof}
		Straightforward.
	\end{proof}
	
	\begin{definition}
		Let $(M,*,0,\leq)$ be an ordered unital magma and $X$ an $M$-metric space. A sequence $(x_n)$ is bounded if there is an element $a\in X$ and there is a positive element $\epsilon$ in $M$ such that \[d(x_n,a) < \epsilon, \qquad \forall~n \in \mathbb N.\]
	\end{definition}
	
	\begin{proposition}\label{convergentisbounded}
		Let $(M,*,0,\leq)$ be an ordered unital magma such that $(M,\leq)$ is a join-semilattice. In an $M$-metric space $X$, a convergent sequence is bounded.
	\end{proposition}
	
	\begin{proof}
		Let $(x_n)$ be a convergent to $a\in X$. Fix a positive element $\epsilon$ of $M$. Then, there is a natural number $N$ such that $n \geq N$ implies $d(x_n,a) < \epsilon$. Set \[R = \sup\{d(x_1,a), \dots,d(x_{N-1},a), \epsilon\}.\] It is now clear that $d(x_i,a) \leq R$ for each $i \in \mathbb N$ showing that $(x_n)$ is bounded and the proof is complete.
	\end{proof}
	
	\begin{definition}\label{Cauchysequencedef}
		Let $(M,*,0,\leq)$ be an ordered unital magma and $X$ an $M$-metric space. A sequence $(x_n)$ in $X$ is a Cauchy sequence if for any positive element $\epsilon$ in $M$ there is a natural number $N$ such that $m,n \geq N$ implies $d(x_m, x_n) < \epsilon$. 
	\end{definition}
	
	\begin{proposition}
		Let $(M,*,0,\leq)$ be an ordered unital magma such that $(M,\leq)$ is a join-semilattice. In an $M$-metric space $X$, a Cauchy sequence is bounded.
	\end{proposition}
	
	\begin{proof}
		Fix a positive $\epsilon \in M$. By definition, we can find a natural number $N$ such that $m,n \geq N$ implies $d(x_m,x_n) < \epsilon$. In particular, we have $d(x_m,x_N) < \epsilon$ for all $m \geq N$. Set \[R = \sup\{d(x_1,x_N), d(x_2,x_N), \dots,d(x_{N-1},x_N), \epsilon\}.\] Then, we see that $d(x_i, x_N) \leq R$ for all $i \in \mathbb N$ showing that $(x_n)$ is bounded in $X$. This completes the proof.
	\end{proof}
	
	\begin{theorem}\label{convergentisCauchy}
		Let $(M,*,0,\leq)$ be a dense unital magma and $X$ an $M$-metric space. Then, any convergent sequence in $X$ is a Cauchy sequence.
	\end{theorem}
	
	\begin{proof}
		Let $(x_n)$ be convergent to $a$. Then, for each positive element $\epsilon \in M$, there is a natural number $N$ such that $n \geq N$ implies that $d(x_n,a) < \epsilon$. For the positive element $\epsilon$, one can find two positive elements $\beta$ and $\gamma$ such that $\beta * \gamma < \epsilon$. On the other hand, for $\beta$ and $\gamma$, one can find natural numbers $N_1$ and $N_2$, respectively, such that if $m \geq N_1$ and $n \geq N_2$, then \[d(x_m,a) < \beta \text{~and~} d(x_n,a) < \gamma.\] Now, observe that \[d(x_m,x_n) \leq d(x_m,a)*d(x_n,a) \leq \beta * \gamma < \epsilon\] whenever $n \geq \max\{N_1,N_2\}$. This shows that $(x_n)$ is a Cauchy sequence and the proof is complete. 
	\end{proof}
	
	\begin{theorem}\label{CauchyconvergentsubsequenceOF}
		Let $(M,*,0,\leq)$ be a dense unital magma. In an $M$-metric space $X$, a Cauchy sequence having a convergent subsequence is convergent.	
	\end{theorem}
	
	\begin{proof}
		Let $(x_n)$ be a Cauchy sequence in $X$. Also, assume that a subsequence $(x_{n_k})$ of $(x_n)$ is convergent to $x \in X$. Let $\epsilon$ be a positive element in $M$. Since $M$ is dense, we can find two positive elements $\beta$ and $\gamma$ with $\beta * \gamma < \epsilon$. For $\beta$, there is a natural number $N_1$ such that for $m,n \geq N_1$, we have $d(x_m,x_n) < \beta$. Also, for $\gamma$, there is a natural number $k_1$ such that $k \geq k_1$ implies that $d(x_{n_k},x) < \gamma$. Note that since $n_k$ is a strictly increasing sequence of positive integers, we have $n_k \geq k$ for each positive integer $k$. Set $N = \max\{N_1,k_1\}$ and observe that for any $m \geq N$ and $k \geq N$, we have \[d(x_m, x) \leq d(x_m , x_{n_k}) * d(x_{n_k}, x) \leq \beta * \gamma < \epsilon.\] Hence, $(x_n)$ is convergent, as required.
	\end{proof}
	
	\section{$M$-normed groups}\label{sec:normedgroups}
	
	\begin{definition}\label{Mnormedgroupdef}
		Let $(M,+,0,\leq)$ be an ordered unital magma. A (not necessarily abelian) group $(G,+)$ is an $M$-normed group if there is a function $\Vert \cdot \Vert: G \rightarrow M$ with the following properties:
		\begin{enumerate}
			\item $\Vert g \Vert \geq 0$, and $\Vert g \Vert = 0$ if and only if $g = 0$, for all $g\in G$,
			\item $\Vert g-h \Vert \leq \Vert g \Vert + \Vert h \Vert$, for all $g,h \in G$.
		\end{enumerate}
	\end{definition}
	
	\begin{proposition}\label{pseudonormevencontinuous}
		Let $M$ be an ordered unital magma and $G$ an $M$-normed group. Then, the following statements hold: \begin{enumerate}
			\item $\Vert -g \Vert = \Vert g \Vert$, for all $g\in G$.
			\item $\Vert g \Vert  \leq \Vert g-h \Vert + \Vert h \Vert$, for all $g,h \in G$.
		\end{enumerate}	
	\end{proposition}
	
	\begin{proof}
		(1): For any $g\in G$, observe that \[\Vert -g \Vert = \Vert 0-g \Vert \leq \Vert 0 \Vert + \Vert g \Vert = 0 + \Vert g \Vert = \Vert g \Vert.\] On the other hand, from this we obtain that \[\Vert g \Vert = \Vert -(-g) \Vert \leq \Vert -g \Vert.\] This proves the first statement. 
		
		(2): Let $g$ and $h$ be elements of $G$. In view of the first statement, observe that \[\Vert g \Vert = \Vert g-0 \Vert = \Vert (g-h)-(-h) \Vert \leq \Vert g-h \Vert + \Vert h \Vert.\] This finishes the proof. 
	\end{proof}
	
	\begin{corollary}\label{reversetriangleinequality}
		Let $M$ be an ordered group and $G$ an $M$-normed group. Then, \[\Vert g \Vert -\Vert h \Vert\leq \Vert g-h \Vert, \qquad\forall~ g,h \in G.\]
	\end{corollary}
	
	\begin{proof}
		Observe that $\Vert g \Vert \leq \Vert g-h \Vert + \Vert h \Vert$. Since $M$ is a group, by adding $- \Vert h \Vert$ to the both sides of the latter inequality (from the right side), the desired inequality is obtained.
	\end{proof}
	
	\begin{proposition}\label{inducedmagmavaluedmetricspace}
		Let $(M,+,0,\leq)$ be an ordered unital magma and $(G,+)$ an $M$-normed group. Then, $(G,d)$ is an $M$-metric space, where \[d(g,h) = \Vert g - h \Vert.\]
	\end{proposition}
	
	\begin{proof}
		Let $g$, $h$, and $k$ be elements of the group $G$. Observe that \[d(g,k) = \Vert g-k \Vert = \Vert (g-h)+(h-k) \Vert \leq \Vert g-h \Vert + \Vert h-k \Vert.\] The proof of the other properties is straightforward. Hence, $(G,d)$ is an $M$-metric space, as required.
	\end{proof}
	
	\begin{proposition}\label{convergentisboudedinnormedgroups}
		Let $M$ be an ordered monoid, $(M,\leq)$ a join-semilattice, and $(x_n)$ a sequence in an $M$-normed group $G$. Then, the following statements hold:
		\begin{enumerate}
			\item If $(x_n)$ is convergent to $0\in G$, then there is a positive element $s$ in $M$ such that $\Vert x_n \Vert \leq s$, for all $n \in \mathbb N$.
			\item If $M$ is a group and $(x_n)$ is convergent to $a$, then there is a positive element $t$ in $M$ such that $\Vert x_n \Vert \leq t$.  
		\end{enumerate}
	\end{proposition}
	
	\begin{proof}
		(1): Since $(x_n)$ converges to 0, by Proposition \ref{convergentisbounded}, there is a positive element $s\in M$ such that \[\Vert x_n \Vert = \Vert x_n - 0 \Vert \leq s, \qquad\forall~ n\in \mathbb N.\]
		(2): Let $(x_n)$ be a sequence in $G$ convergent to $a\in G$. Then, by Proposition \ref{convergentisbounded}, there is a positive element $s\in M$ such that \[\Vert x_n - a \Vert \leq s, \qquad\forall~ n\in \mathbb N.\] Since $M$ is a group, by Corollary \ref{reversetriangleinequality}, we have \[\Vert x_n \Vert - \Vert a \Vert \leq \Vert x_n - a \Vert \leq s.\] Set $t = s + \Vert a \Vert$. Then, we have $\Vert x_n \Vert \leq t$. This finishes the proof.
	\end{proof}
	
	Let $(G,+)$ be an abelian group and $\Seq(G)$ the set of all sequences over $G$. It is clear that $\Seq(G)$ with component-wise addition is an abelian group. 
	
	\begin{theorem}\label{Convergentsequencesabeliangroup}
		Let $(M,*,0,\leq)$ be a dense unital magma and $G$ an $M$-normed abelian group. Then, the following statements hold:
		
		\begin{enumerate}
			\item The set of all Cauchy sequences $\Cauchy(G)$ of $G$ is a subgroup of $\Seq(G)$.
			\item\label{addconvsequencesconv} If $(x_n)$ and $(y_n)$ are sequences in $G$ convergent to $a$ and $b$ in $G$, respectively, then the sequence $(x_n + y_n)$ is convergent to $a+b$.
			\item The set of convergent sequences, denoted by $\Conv(G)$, is a subgroup of $\Cauchy(G)$.
			\item If $\Zero(G)$ is the set of all sequences in $\Conv(G)$ which are convergent to $0 \in G$, then $\Conv(G)/\Zero(G) \cong G$.
		\end{enumerate}
	\end{theorem}
	
	\begin{proof}
		(1): Let $(x_n)$ and $(y_n)$ be Cauchy sequences. Since $M$ is dense, for the given positive $\epsilon$ in $M$, we can find two positive elements $\beta$ and $\gamma$ such that $\beta * \gamma < \epsilon$. Also, we can find a natural number $N$ such that $m,n \geq N$ implies the following: \[\Vert x_m - x_n \Vert < \beta \text{~and~} \Vert y_m - y_n \Vert < \gamma.\] Now, observe that if $m,n \geq N$, we have \[\Vert (x_m + y_m) - (x_n + y_n) \leq \Vert x_m - x_n \Vert * \Vert y_m - y_n \Vert \leq \beta * \gamma < \epsilon.\] This shows that $(x_n + y_n)$ is a Cauchy sequence. It is obvious that the constant sequence $(0)$, where $0\in G$ is the identity element of $G$, is Cauchy, and if $(x_n)$ is Cauchy, then so is $(-x_n)$. Thus $\Cauchy(G)$ is a subgroup of $\Seq(G)$.
		
		(2): For the given positive element $\epsilon \in M$, we can find positive elements $\beta$ and $\gamma$ such that $\beta * \gamma < \epsilon$. It is evident that one can find a natural number $N$ such that if $n \geq N$, then \[\Vert x_n - a \Vert < \beta \text{~and~} \Vert y_n - b \Vert < \gamma.\] Now, observe that \[\Vert (x_n + y_n) - (a+b) \Vert \leq \Vert x_n - a \Vert * \Vert y_n - b \Vert \leq \beta * \gamma < \epsilon.\] This shows that $(x_n + y_n)$ is convergent to $a+b$. 
		
		(3): By Theorem \ref{convergentisCauchy}, each convergent sequence is a Cauchy sequence. So, $\Conv(G)$ is a subset of $\Cauchy(G)$. By (\ref{addconvsequencesconv}), $\Conv(G)$ is closed under addition of sequences. On the other hand, the constant sequence $(0)$ is in $\Conv(G)$, and, if $(x_n)$ is convergent to $a\in G$, then $(-x_n)$ is convergent to $-a$. Consequently, $\Conv(G)$ is a subgroup of $\Cauchy(G)$.
		
		(4): Define $\varphi: \Conv(G) \rightarrow G$ by $(x_n) \mapsto \lim_{n \to +\infty} x_n$. By (\ref{addconvsequencesconv}), $\varphi$ is a group homomorphism. For any $g\in G$, set $g_n  = g$. Then, by Example \ref{limitofconstantsequences}, $\varphi(g_n) = g$ showing that $\varphi$ is an epimorphism. The kernel of $\varphi$ is the set of all sequences that are convergent to $0 \in G$. Hence, by the fundamental theorem of homomorphisms, we have \[\Conv(G)/\Zero(G) \cong G,\] and the proof is complete.
	\end{proof}

	\begin{theorem}\label{Cauchynotconvergenttozerothm}
		Let $(M,+,0,\leq)$ be a dense and a totally ordered group and $(G,+)$ an $M$-normed group. Let $(x_n)$ be a Cauchy sequence not convergent to zero. Then, there is a positive element $\gamma$ in $M$ and a natural number $N$ such that $n \geq N$ implies $\Vert x_n \Vert > \gamma$.
	\end{theorem}
	
	\begin{proof}
		Since $(M,+,0,\leq)$ is a totally ordered group and $(x_n)$ is not convergent to zero, there is a positive element $\epsilon$ in $M$ such that for each $n \in \mathbb N$, there is a natural number $k$ such that \begin{align}
			k \geq n \text{~and~} \Vert x_k \Vert \geq \epsilon. \label{Cauchynotconvergenttozerothm1}
		\end{align}  By assumption $M$ is dense. Therefore, by Lemma \ref{densemonoidlem}, we can find a positive element $\beta$ in $M$ with $\beta < \epsilon$. Since $(x_n)$ is a Cauchy sequence, we can find a natural number $N$ such that \begin{align}
			\Vert x_p - x_q \Vert < \beta \qquad \forall~p,q \geq N. \label{Cauchynotconvergenttozerothm2}
		\end{align} Assume that $n$ is an arbitrary positive integer with $n \geq N$. By (\ref{Cauchynotconvergenttozerothm1}) and (\ref{Cauchynotconvergenttozerothm2}), there is a natural number $k$ such that \[\Vert x_k \Vert \geq \epsilon \text{~and~} -\Vert x_k - x_n \Vert > -\beta.\] Now, observe that for all $n \geq N$, we have \[\Vert x_n \Vert \geq - \Vert x_k - x_n \Vert + \Vert x_k \Vert  > - \beta + \epsilon  > 0.\] Let $\gamma$ be a positive element of $M$ smaller than $\epsilon - \beta$. Hence, there is a natural number $N$ such that $n \geq N$ implies $\Vert x_n \Vert > \gamma$, as required.
	\end{proof}
	
	\section*{Acknowledgments} The author is grateful to Prof. Winfried Bruns for his encouragements and the referees for their comments which improved the presentation of the paper.

\end{document}